\documentclass[12pt]{amsart}

\usepackage{bbold}
\usepackage{amsmath}
\usepackage{amsthm}
\usepackage{graphicx}
\usepackage{amssymb}
\usepackage{xcolor}

\usepackage [english]{babel}
\usepackage [autostyle, english = american]{csquotes}
\MakeOuterQuote{"}

\newtheorem{thm}{Theorem}[section]
\newtheorem{lemma}[thm]{Lemma}

\newtheorem{corollary}[thm]{Corollary}

\theoremstyle{definition}

\theoremstyle{remark}

\newcommand{\dfn}[1]{{\bf #1}\index{#1}}

\title[Random lower triangular matrices]{Macroscale behavior of random lower triangular matrices}
\author{
J. E. Pascoe \\
\\
Tapesh Yadav
}
\date{\today}
\thanks{Pascoe was partially supported by NSF-DMS Analysis grant 1953963.}
\begin{document}


\begin{abstract}
    We analyze the macroscale behavior of random lower (and therefore upper) triangular matrices with entries drawn iid from a  distribution with nonzero mean and finite variance. We show that such a matrix behaves like a probabilistic version of a Riemann sum and therefore in the limit behaves like the Volterra operator. Specifically, we analyze certain SOT-like and WOT-like modes of convergence for random lower triangular matrices to a scaled Volterra operator. We close with a brief discussion of moments.
\end{abstract}
\maketitle

\section{Introduction }

 The Wigner semicircle law states that a class of self adjoint random matrices called Wigner matrices go to semicircular element $a.s$ and in distribution asymptotically.
Specifically, if one considers a large random Hermitian matrix with entries drawn i.i.d. from a suitably nice distribution, when we look at the histogram of the eigenvalues, we see a semicircular shape, with perhaps one large exceptional eigenvalue.
The theory of free probability and random matrix theory give various ways in which we can make this convergence formal \cite{speicherbook}.

\begin{figure}[!]
    \centering
\includegraphics[width=12cm, height=5.5cm]{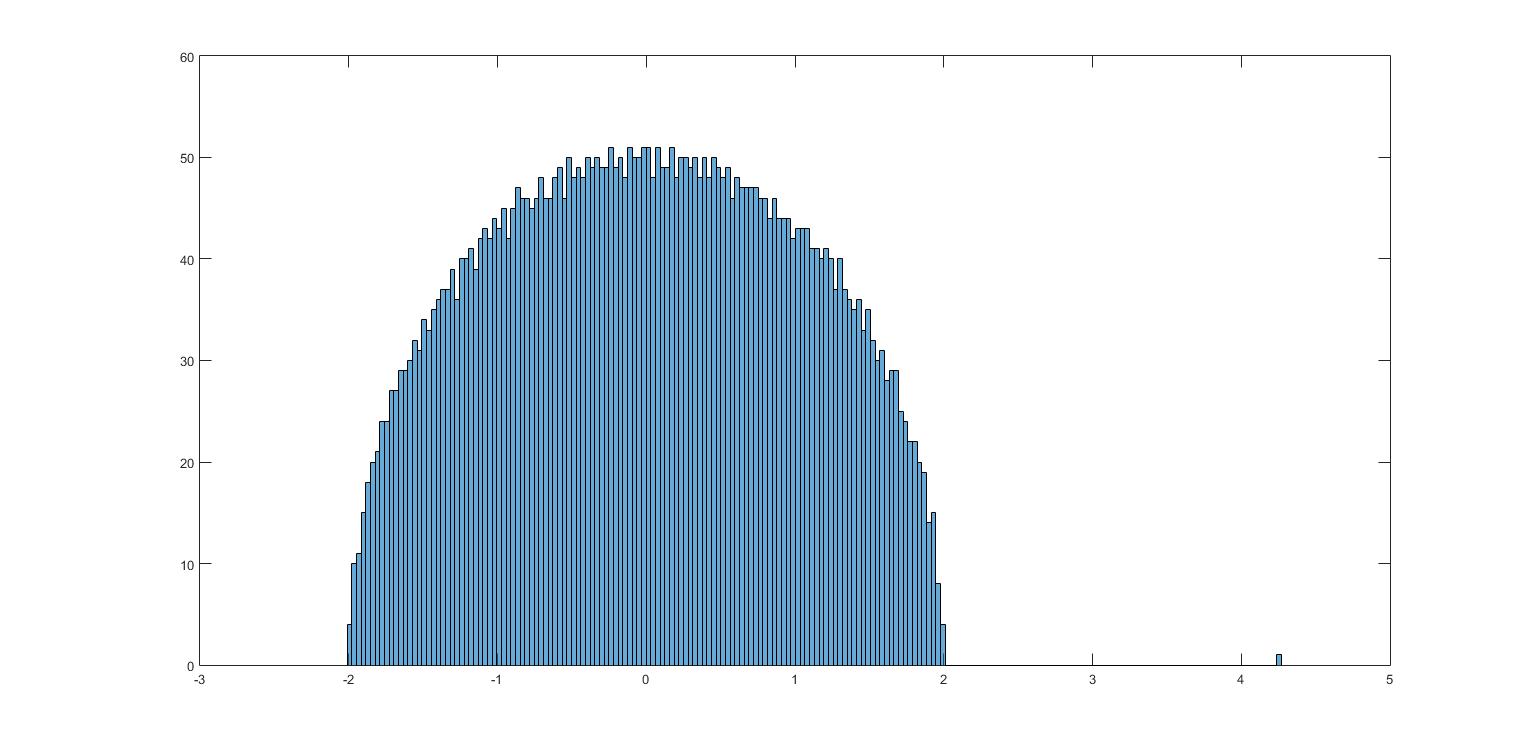}
    \caption{a  histogram of 5000 by 5000 Gaussian self adjoint (real) random matrix (GOE) with non zero mean. Each entry has mean 4/N and standard deviation $\frac{1}{\sqrt{N}}$, N=5000. Note there is one exceptional large eigenvalue, while the rest follow a semi-circular distribution.}
     \label{fig0}
\end{figure}

In their breakthrough paper, \cite{dykema}, Dykema and Haagerup looked at distribution limits of their upper triangular random matrices with iid complex Gaussian entries having mean zero  and variance $c^2/n$ in the strictly upper triangular part and iid random variables distributed according to a compactly supported measure $\mu$ on the main diagonal. The limiting non-commuting random variable exists and is called a DT-element or DT-operator. The DT-operators include Voiculescu's circular operator and elliptic deformations of it, as well as the circular free Poisson operators. Star moments of of these operators show interesting combinatorial properties, as is explored in \cite{sniady} by Sniady.  Dykema and Haagerup \cite{dykemainv} later proved that every DT operator has a nontrivial, closed, hyperinvariant subspace. Furthermore, every DT-operator generates the von Neumann algebra $L(\mathbb{F}_2)$ of the free group on two generators.\\

Let $X_N=\frac{1}{N}(X_{i,j}^N)_{i,j=1}^N$ denote an $N \times N$ random lower triangular matrix where $X_{i,j}^N$'s are iid random variables with finite mean $\mu_N$ and finite variance $\sigma_N^2$. Let $T_N$ be the deterministic lower triangular matrix with each entry being $1/N$ i.e. 

$$X_N= \frac{1}{N}
\begin{bmatrix}
X_{1,1}^N & 0 & \ldots & 0\\
X_{2,1}^N & X_{2,2}^N & \ldots & 0\\
\hdots & \hdots & \ddots & \hdots\\
X_{N,1}^N & X_{N,2}^N & ... & X_{N,N}^N
\end{bmatrix},  T_N= \frac{1}{N}
\begin{bmatrix}
1 & 0 & \ldots & 0\\
1 & 1 & \ldots & 0\\
\hdots & \hdots & \ddots & \hdots\\
1 & 1 & ... & 1
\end{bmatrix}\\ \\$$
\\

We ran some experiments to see the singular value distribution for large random matrices. Fig. \ref{fig1} shows singular value distribution of $X_N = \frac{\pi}{N} (X_{i,j}^N)_{i,j=1}^N$ for $N=5000$, where $X_{i,j}^N$ are iid Bernoulli(0,1) random variable for $1 \leq j \leq i \leq N$ and 0 otherwise. Singular values of $X_N$ for large N behave like the singular value distribution for DT operators near 0 and like the Volterra operator away from 0. Our current investigation only concerns the asymptotic description of large singular values. \\

\begin{figure}[!] 
    \centering
\includegraphics[width=12cm, height=5.5cm]{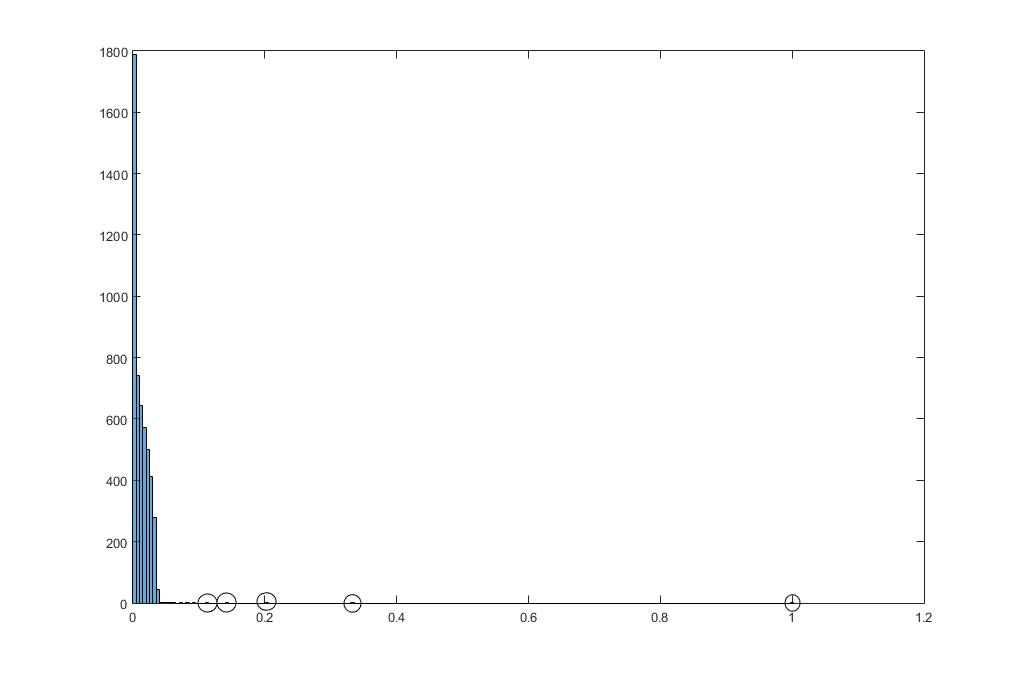}
    \caption{singular value distribution of $X_N = \frac{\pi}{N} (X_{i,j}^N)_{i,j=1}^N$ for $N=5000$, where $X_{i,j}^N$ is iid Bernoulli(0,1) random variable for all $1 \leq j \leq i \leq N$ and 0 otherwise.}
    \label{fig1}
\end{figure}

Let $V$ be the \dfn{Volterra operator} on $L^2[0,1]$ defined by 
$$V(f)(x)= \int_0^x f(t)dt$$
for all $f \in L^2[0,1]$. Let $W_N:\mathbb{C}^N \to L^2[0,1]$ be the isometry taking a vector to a piecewise constant function,  as formally defined in the next section. \\


\begin{thm}[SOT-like convergence]\label{thm}\label{sotthm}
    Let  $\{a_N\}_{N=0}^\infty \subset \mathbb{Z}^+$ be a non-negative increasing sequence.
    Let $\{k(N)\}$ be a sequence of non-negative real numbers such that $\frac{k(N) \sigma_N}{\sqrt{a_N}} \to 0$ and  $\sum_{N=1}^\infty \frac{1}{k(N)^2} < \infty$. Let $\mu_{a_N} \to \mu$.
    
    Then, for all $u \in L^2[0,1]$, $$ W_{a_N}X_{a_N}W_{a_N}^*(u) \rightarrow \mu V(u) \textrm{ a.s.}$$   
    \end{thm}

For instance, we have that $W_{2^N}X_{2^N}W_{2^N}^*(f) \to \mu V(f) $ a.s. for all $f \in L^2[0,1]$ whenever $\mu_N \to \mu$ and the standard deviations are uniformly bounded. We discuss important properties of SOT-like convergence and prove Theorem \ref{sotthm} in Section \ref{sotsect}. The idea is that the $T_N$ act on vector in $\mathbb{C}^N$ consisting of function values taken from equally distanced points in interval $[0,1]$ and outputs the partial sum for that function, which converge to integral of the function as in a Riemann sum. The matrix $T_N$ have singular values similar to Volterra operator, which are $\frac{2}{\pi(2n+1)}$. 

We also have a WOT version of this theorem, which requires considerably weaker conditions for convergence.
\begin{thm}[WOT-like convergence]\label{wotthm}
    Let  $\{a_N\}_{N=0}^\infty \subset \mathbb{Z}^+$ be a non-negative increasing sequence.
    Let $\{k(N)\}$ be a sequence of non-negative real numbers such that $\frac{k(N) \sigma_N}{a_N} \to 0$ and  $\sum_{N=1}^\infty \frac{1}{k(N)^2} < \infty$. Let $\mu_{a_N} \to \mu$.
    
    Then, for all $u, v \in L^2[0,1]$, $$\langle W_{a_N}X_{a_N}W_{a_N}^*(u), v\rangle \rightarrow \mu \langle V(u), v\rangle \textrm{ a.s.}$$   
\end{thm}
For instance, we can conclude WOT-like convergence along the sequence $X_N$ whenever $\mu_N \to \mu$ and the standard deviations are uniformly bounded. (Specifically, there is not enough variance to neccesitate taking a subsequence as in Theorem \ref{sotthm}.)
We discuss WOT-like convergence and prove Theorem \ref{sotthm} in Section \ref{sotsect}.

Also, one can remove the term `\textit{like}' from the above definitions if the random matrices under consideration ($X_N$'s) are uniformly bounded in operator norm a.s. For example, the $X_N$ will be uniformly bounded for
Bernoulli $0-1$ random variables with fixed mean and variance.
\\

In the last section, we give moment results for $X_N^*X_N$ for any random matrix with finite moments for each entry and of $NX_N^*X_N$ in the case of non-zero mean. The zero mean case was studied by Dykema and Haagerup \cite{dykema} where each entry of $X_N$ was Gaussian. We do not see any direct way to generalize their method to matrix $X_N$ with non-Gaussian random variables. Also, in non-zero mean case, as Figure \ref{fig1} suggest, we do not get a copy of mean zero spectrum with an exceptional eigenvalue, as is the case in non zero mean Wigner matrices. Our empirical observations show a superimposition of singular values from the Volterra operator and DT operator.

\section{SOT-like convergence} \label{sotsect}

Let $$B_N = \overline{span} \{e_1^N, \ldots, e_N^N\} \subseteq L^2[0,1],$$  where for $1\leq i \leq N  $, the function $e_i^N= $ $\sqrt{N} \textbf{1}_{[(i-1)/N,i/N]}$ and $\textbf{1}_{[(i-1)/N,i/N]}$ is the indicator function of the interval $[(i-1)/N,i/N]$. Note that $\{e_i^N\}_{i=1}^N$ form an orthonormal basis for $B_N$. We define the isometry $W_N:\mathbb{C}^N \to L^2[0,1]$ by $$W_N(a_1, ...,a_N) = \sum_{i=1}^n a_ie_i^N.$$ $W_N$ takes $\mathbb{C}^N$ onto $B_N$ isometrically. Note that $W_N^*$ is a partial isometry which sends $f$ to $(\langle f,e_1^N \rangle, ... , \langle f,e_N^N\rangle)$. \\

Let us begin with the following useful lemma.

\begin{lemma} \label{1.2}
		Let V be the Volterra operator on $L^2[0,1]$. Then, $W_{N}T_{N}W_{N}^* \to V $ in SOT.
	\end{lemma} 

\begin{proof}
Let $$g_N(x)= W_NT_NW_N^*(f)(x).$$ We first show $g_N \to V(f)$ pointwise for each $f \in C([0,1])$. Without loss of generality, consider a non negative continuous function $f \in C([0,1])$. There exists $x_i^N \in [(i-1)/N,i/N]$ such that $f(x_i^N)=\sqrt{N}\langle f,e_1^N\rangle $ by intermediate value theorem. Define $a_N(x)= min\{i \in \mathbb{N}: x \leq i/N\}$. For fixed $x \in [0,1]$, $$(1/\sqrt{N})\sum_{i=1}^{a_N(x)} \langle f,e_i^N \rangle  \to \int_0^x f(t)dt$$ as $N \to \infty$. Thus $g_N(x) \to \int_0^x f$ $ = V(f)(x)$ pointwise. Therefore $$g_N(x) \to \int_0^x f(t)dt$$ in $L^2[0,1]$ by the bounded convergence theorem (every function is bounded by the sup norm of $f$).  Hence, $$lim_{N \to \infty} ||(W_NT_NW_N^* - V)(f)|| \to 0.$$  Since, we obtain convergence for all continuous functions on $[0,1],$ which are dense in $L^2[0,1]$,
and
the norms of $\{W_NT_NW_N^*\}$ and $V$ are uniformly bounded by $2,$ we have that, $||(W_NT_NW_N^* - V)(f)|| \to 0$  for all $f \in L^2[0,1]$.
\end{proof}

 For $u=(u_1,u_2,...,u_N) \in \mathbb{C}^N$, let $u^2$ denote the vector $(|u_1|^2,|u_2|^2,...,|u_N|^2) \in \mathbb{C}^N$. 
      \begin{lemma}
        Let $u$ be a unit vector in $\mathbb{C}^N$, then
        $E((\mu_NT_N-X_N)u)=0$ and
        $E((\|(\mu_NT_N-X_N)u\|)^2)\leq \|\frac{\sigma_N^2}{N}T_Nu^2\|_1$
    \end{lemma} 
    
    \begin{proof}
    The first equality is direct. For the second inequality, observe that
    \begin{align*}
        E((\|(\mu_NT_N-X_N)u\|)^2) &= \frac{1}{N^2}E(\sum_{i=1}^N |\sum_{j=1}^i (X_{i,j} - \mu)u_j|^2) \\
        & \leq \frac{1}{N^2}E(\sum_{i=1}^N \sum_{j=1}^i |(X_{i,j} - \mu)u_j|^2) \\
        & = \frac{1}{N^2}\sum_{i=1}^N \sum_{j=1}^i E(|X_{i,j} - \mu|^2)|u_j|^2 \\
        & = \frac{1}{N^2}\sum_{i=1}^N \sum_{j=1}^i \sigma^2|u_j|^2 \\
        &= || \frac{\sigma^2}{N} T_N u^2 ||_1 
    \end{align*}
    \end{proof}
    
    For a non-negative sequence ${k(N)}$, Chebychev's inequality implies that
        \begin{equation}\label{1}
        P\bigg(\|(\mu_NT_N-X_N)u\|\geq k(N)\sqrt{\|\frac{\sigma_N^2}{N}T_Nu^2\|_1} \bigg)\leq \frac{1}{k(N)^2}.
        \end{equation}
    Therefore, we can finesse our estimate for the standard deviation into a statement about almost sure convergence.    
        \begin{lemma}
     Let $\{a_N\}_{N=0}^\infty \subset \mathbb{Z}^+$ be a non-negative increasing sequence. If there exists positive sequence $\{k(N)\}$ such that $\frac{k(N) \sigma_N}{\sqrt{a_N}} \to 0$ and $\sum_{N=1}^\infty \frac{1}{k(N)^2} < \infty$. Then $\|W_{a_N}(\mu_{a_N}T_{a_N}-X_{a_N})W_{a_N}^*u\|$ $\rightarrow 0 \textrm{ a.s.}$ for all $u \in L^2[0,1].$ 
    \end{lemma}
    
   \begin{proof}
    From Eq. (\ref{1}), we get that, 
        $$P(\|(\mu_{a_N}T_{a_N}-X_{a_N})u\|\geq k(N) \sqrt{\|\frac{\sigma_N^2}{a_N}T_{a_N}u^2\|_1})\leq \frac{1}{k(N)^2}.$$ 
         The right hand side is summable. So, by the Borel-Cantelli lemma, the probability that the events \{$\|(\mu_{a_N}T_{a_N}-X_{a_N})u\|\geq k(N)  \sqrt{\|\frac{\sigma_N^2}{a_N}T_{a_N}u^2\|_1}$\}, occur infinitely often is 0. Observe that for a unit vector $u$, if $u_{a_N}$ denotes $W_{a_N}^*u$, then $u_{a_N}$ has norm less than or equal to 1 (as $W^*$ is projection). So, $||T_{a_N}u_{a_N}^2||_1 \leq 1$.  hence $k(N) \sqrt{\|\frac{\sigma_N^2}{a_N}T_{a_N}u^2\|_1} \leq \frac{k(N) \sigma_N}{\sqrt{a_N}}$. So, \\ 
       $$ \|(\mu_{a_N}T_{a_N}-X_{a_N})u_{a_N}\| \leq \frac{k(N) \sigma_N}{\sqrt{a_N}}  \text{ eventually a.s.}$$ 
       This gives that,
      $$ \|(\mu_{a_N}T_{a_N}-X_{a_N})u_{a_N}\| \to 0 \text{ a.s.}$$ 
      $$ \implies \|(\mu_{a_N}T_{a_N}-X_{a_N})W^*_{a_N}u\| \to 0 \text{ a.s.}$$
      $$ \implies \|W_{a_N}(\mu_{a_N}T_{a_N}-X_{a_N})W^*_{a_N}u\| \to 0 \text{ a.s.}$$ 
       This is true for any unit vector $u$, and hence for any vector in general. 
         \end{proof}


    \begin{proof}[\textbf{Proof of Theorem \ref{sotthm}}]
    Lemma \ref{1.2} along with triangle inequality gives that  $\| (\mu V - \mu_{a_N} W_{a_N}T_{a_N}W_{a_N}^*)u\| \to 0$ for all $u \in L^2[0,1]$. Hence, 
    \begin{align*}
        & \|(\mu V - W_{a_N}X_{a_N}W_{a_N}^*)u\| \\  
        & \leq \| (\mu V - \mu_{a_N} W_{a_N}T_{a_N}W_{a_N}^*)u\| + \|W_{a_N}(\mu_{a_N}T_{a_N}-X_{a_N})W_{a_N}^*u\| \\ 
        & \to 0 \text{ a.s} 
    \end{align*}
    \end{proof}

\subsection{Remarks on SOT convergence}

 \begin{enumerate}
    \item The above theorem is rather powerful. For example, if the variance goes to 0 at a rate faster than $\frac{1}{N^{\epsilon}}$ for some $\epsilon > 0$, then we have guaranteed convergence for any sequence $a_N$. In particular for $a_N=N$ which gives $W_NX_NW_N^*(u) \to \mu V(u)$ a.s. (choose $k(N)= N^{\frac{1+\epsilon}{2}}.$ ).
        \item The sequence $\{k(N)\}$ may not exist in some cases. For example, let $\sigma_N=\sigma$ be constant. Then if $a_N=N$, we do not have any sequence which achieves the goal. This implies that if all the random variables comes from the same distribution independent of size of matrix N, then the above theorem cannot guarantee convergence to the Volterra operator for the random matrices. 
        \item If norm of random matrices can be bounded uniformly a.s. then we can conclude true SOT convergence. 
        
    \end{enumerate}
    
    An important case for convergence (for $a_N=2^N$) can be seen in the corollary below.
    \begin{corollary}
    $\forall u \in L^2[0,1]$ $\|(\mu V -W_{2^N}X_{2^N}W_{2^N}^*)u\| \rightarrow 0$ a.s. whenever for $j \leq i \leq N$, $X_{i,j}^N$ are iid random variables (independent of N) with mean $\mu$ and finite variance $\sigma$. 
    \end{corollary}
    
    \begin{proof} 
    Choose $k(N) = 2^{N/4}$
    \end{proof}

    \section{WOT-like convergence}
    
    Let $X_N$ be as earlier. We have the following variance bound.

    \begin{lemma} Let $u,v$ be vectors in $\mathbb{C}^N$, then
        $$E(\langle( X_N- \mu_N T_N )u, v\rangle ) = 0$$ and,
        $$E(|\langle( X_N- \mu_N T_N) u, v\rangle|^2) \leq \frac{\sigma_N^2}{N^2}\sum_{i=1}^N \sum_{j=1}^i |v_i|^2|u_j|^2  \leq \frac{\sigma_N^2}{N^2} \|u\|^2\|v\|^2$$
    \end{lemma}
    
    \begin{proof}
    First equality is direct. The second inequality is also direct after expanding and using triangle inequality.
    \end{proof}

    So
    \begin{equation}\label{borcantwot}
         P\bigg(|\langle (X_N-\mu_N T_N) u, v\rangle|>k(N) \frac{\sigma_N}{N}\|u\|\|v\|\bigg)\leq \frac{1}{k(N)^2}
        \end{equation}
    and therefore we obtain Theorem \ref{wotthm} via a similar argument to the proof of Theorem \ref{sotthm}. 
    
    Equation \eqref{borcantwot} gives us that unlike the SOT-like case, whenever $\mu_N \to \mu$, we do have $\langle W_{N}X_{N}W_{N}^*u,v \rangle \rightarrow \langle \mu Vu,v \rangle$ a.s., whenever $\{\sigma_N\}$ is uniformly bounded.\\
    
     Let $X_N$ be a random lower triangular matrix such that an entry is $\frac{1}{\delta (N) N}$ with probability $\delta (N)$ and $0$ otherwise. This gives mean, $\mu_N=1$ and variance, $\sigma_N^2 = \frac{1-\delta (N)}{\delta (N)}$. Then, 
     \begin{equation}
        P\bigg(|\langle (X_N-T_N) u, v\rangle|>k(N) \sqrt{\frac{1-\delta (N)}{\delta (N) N^2}}\|u\|\|v\| \bigg)\leq \frac{1}{k(N)^2}
    \end{equation}
    
    \begin{itemize}
        \item  If $\delta (N)$ is bounded below uniformly, then $\langle W_NX_NW_N^*u,v \rangle \to \langle Vu,v \rangle$ a.s. (Choose $k(N)=N^{1/2+\varepsilon}$).\\
        \item If $\delta(N)= N^{-d}$, and $d<1$, we can show that we still have WOT-like convergence (choose $k(N)= N^{\frac{(3-d)}{4}}$). If $d \geq 1$, theorem \ref{wotthm} cannot guarantee WOT like convergence.
        
    \end{itemize}

\section{Asymptotic distribution of $X_N^*X_N$ and $NX_N^*X_N$}

We will begin with the following observation about the deterministic matrix $T_N$. For fixed $1 \leq k \leq N$, let $\mathbb{1}_k^N$ be N by N deterministic matrix with entry  $(\mathbb{1}_k^N)_{ij}= 1$ if $i,j \leq k$ and 0 otherwise.

\vspace{8pt}
\centering{$$\mathbb{1}_k^N = 
\begin{bmatrix}
1 & \ldots & 1 & 0 & \ldots  & 0\\
\vdots & \ddots  & \vdots & \vdots & \ddots & \vdots\\
1 & \hdots  & 1 & 0 & \hdots & 0\\
0 & \ldots & 0 & 0 & \ldots  & 0 \\
\vdots & \ddots & \vdots & \vdots & \ddots & \vdots \\
0 & \ldots & 0 & 0 & \ldots  & 0 \\
\end{bmatrix}$$}

\vspace{8pt}

\begin{lemma} \label{lem1}
$$\lfloor N/2 \rfloor^{2n} \leq Tr(N^2(T_N^*T_N)^n) \leq N^{2n}$$  for all  $n \geq 1.$
\end{lemma}

\begin{proof}
Basic computations show that,
\begin{equation}
    N^2 (T_N^*T_N) = \sum_{k=1}^{N} \mathbb{1}_k^N
\end{equation}

With this piece of information, we can see that,
\begin{align*}
    Tr(N^2 (T_N^*T_N))^n &= Tr(\sum_{k=1}^{N} \mathbb{1}_k^N )^n \nonumber \\ 
    &= \sum_{i_1,...,i_n=1}^N Tr( \mathbb{1}^N_{i_1} \mathbb{1}^N_{i_2}...\mathbb{1}^N_{i_n})\\
    & \leq \sum_{i_1,...,i_n=1}^N Tr( {\mathbb{1}^N_{N}})^n \nonumber \\
    & = \sum_{i_1,...,i_n=1}^N N^n \nonumber \\
    &= N^n (\sum_{i_1,...,i_n=1}^N1)  \nonumber \\
    & = N^n N^n = N^{2n}. \nonumber
\end{align*} 
\\

This gives the upper bound. For lower bound, we observe that we can restrict indices $\lfloor N/2 \rfloor \leq i_l \leq N $ for all $l=1,...,n$. Under this restriction, $$Tr( \mathbb{1}^N_{i_1} \mathbb{1}^N_{i_2}...\mathbb{1}^N_{i_n}) \geq Tr( \mathbb{1}^N_{\lfloor N/2 \rfloor})^n = (\lfloor N/2 \rfloor)^n .$$ So, we get
\begin{align*}
\sum_{i_1,...,i_n=1}^N Tr( \mathbb{1}^N_{i_1} \mathbb{1}^N_{i_2}...\mathbb{1}^N_{i_n}) &
\geq \sum_{i_1,...,i_n=\lfloor N/2 \rfloor}^N Tr( \mathbb{1}^N_{i_1} \mathbb{1}^N_{i_2}...\mathbb{1}^N_{i_n})\\
&\geq (\lfloor N/2 \rfloor)^n (\sum_{i_1,...,i_n=\lfloor N/2 \rfloor}^N 1)\\
&\geq (\lfloor N/2 \rfloor)^n (\lfloor N/2 \rfloor)^n \\
&= (\lfloor N/2 \rfloor)^{2n}. 
\end{align*} This proves the lower bound.
\end{proof}

\begin{lemma}
Let $\{X^N_{ij}\}$ be uniformly bounded by constant K a.s. Then, $tr((X_N^*X_N)^n) \to 0$ a.s as $N \to \infty$ for all $n\geq 1$.
\end{lemma}

\begin{proof}

We observe that, $$tr((X_N^*X_N)^n) \leq tr(K^2T_N^*T_N)^n \text{ } a.s$$. By Lemma \ref{lem1}, we have that 
\begin{align*}
 Tr(K^2T_N^*T_N)^n &\leq K^{2n}\\
 \implies tr(K^2T_N^*T_N)^n &\leq (K^{2n})/N \\
 \implies tr((X_N^*X_N)^n) &\leq (K^{2n})/N \to 0 \text{ } a.s  \text{ as } N \to \infty.    
\end{align*}

\end{proof}





\begin{lemma} \label{lem4}
Let $\{X^N_{ij}\}$ be iid random variables with finite moments. Then, $E[tr((X_N^*X_N)^n)] \to 0$ as $N \to \infty$ for all $n\geq 1$.
\end{lemma}

\begin{proof}

We observe that, after expanding $Tr(N^2(X_N^*X_N))^n$, there are at most $N^{2n}$ terms  of the form $X^N_{i_1j_1}X^N_{i_2j_1}X^N_{i_2j_2}...X^N_{i_nj_n}X^N_{i_1j_n}$ for $i_l,j_k \in {1,...,N}$ and $l,k \in \{1,...,n\}$. Since $X_{ij}$'s are iid, we get that, for fixed $n$, expectation of each term can take value from a finite set of numbers independent of N. For example, it can be $E[X_{11}]^{2n}$, if the pairs $(i_l,j_l)$ and $(i_{l+1},j_l)$ are all distinct, i.e, every random variable is independent of each other in the term. It can be $E[X_{11}^{2n}]$, if $(i_l,j_l)=(i_{l+1},j_{l+1})$ for all $l={1,...,n-1}$, i.e., we have the same random variable multiplied $2n$ times. This gives that there are finitely many values that each term in trace expansion can take. Let $M_n$ be the maximum absolute value in this set. Thus, each term $|X^N_{i_1j_1}X^N_{i_2j_1}X^N_{i_2j_2}...X^N_{i_nj_n}X^N_{i_1j_n}| \leq M_n$ for all $i_l,j_k \in \{1,...,N\}$ and $l,k \in \{1,...,n\}$, independent of $N$. Since there are at most $N^{2n}$ such terms, we have that $E[Tr(N^2(X_N^*X_N)^n)] \leq M_nN^{2n}$, which gives, $E[tr(X_N^*X_N)^n] \leq M_n/N \to 0$. Hence the claim.

\end{proof}

\begin{lemma}
Let $\{X^N_{ij}\}$ be collection of iid random variables with mean, $\mu \neq 0$. Then, $E[tr((NX_N^*X_N)^n)] \to \infty$ as $N \to \infty$ for all $n\geq 2$. For $n=1$, $E[tr(NX_N^*X_N)] \to (\sigma^2 + \mu^2)/2$ as $N \to \infty$.
\end{lemma}

\begin{proof}


If we expand $Tr((NX_N^*X_N)^n)$,  we get that each term is of the form, $N^n(X^N_{i_1j_1}X^N_{i_2j_1}X^N_{i_2j_2}...X^N_{i_nj_n}X^N_{i_1j_n})$ for $i_l,j_k \in {1,...,N}$ and $l,k \in {1,...,n}$. Note that, the term equals 0 if $i_l < j_l$ or $i_{l+1} < j_l$, i.e. $j_l \leq min\{i_l,i_{l+1}\}$. While $i_l$'s are free to take any value from $\{1,...,N\}$ and $j_l,j_{l-1}$ are restricted due to that, we can restrict $i_l \geq \lfloor N/2 \rfloor$ for all $l=1,..,n$. Total such possibilities are at least $(N/2)^n$. Moreover, each $j_l$ is free to take value till $\lfloor N/2 \rfloor$. Number of terms following this constraint are of order $O(N^{2n})$. Also, the number of paths $i_1 \to j_1 \to i_2 \to j_2 ... \to j_n \to i_1$, under the restriction that at least a pair of numbers $i_l,j_k$ is same, is of order $O(N^{2n-1})$.  Hence, terms with all distinct random variables $X_{i_1j_1}X_{i_2j_1}X_{i_2j_2}...X_{i_nj_n}X_{i_1j_n}$ grows as $O(N^{2n})$, while the remaining terms grow at $O(N^{2n-1})$ . If all random variable are distinct, we get that $E[X_{i_1j_1}X_{i_2j_1}X_{i_2j_2}...X_{i_nj_n}X_{i_1j_n}] = \mu^{2n}$. Summing over each such term (number of such terms is bigger that $K_nN^2$ for some positive $K_n$) gives that $E[tr((NX_N^*X_N)^n)] \to \infty$ as $N \to \infty$.\\
For $n=1$, we know that for any matrix $A,$ $Tr(A^*A)$ is equal to the square sum of its entries. So, $E[tr(NX_N^*X_N)] = \frac{1}{N^2} \sum_{i=1}^N \sum_{j=1}^i E[X_{ij}^2]  = \frac{N(N+1)/2}{N^2} (\sigma^2 + \mu^2) \to \frac{(\sigma^2 + \mu^2)}{2}$ as $N \to \infty$.

\end{proof}

	\bibliographystyle{unsrt}
	\bibliography{reference}	

\end{document}